\documentclass{article}
\usepackage{amssymb}
\usepackage{tikz}

\newcommand{\Zer}{\mathrm{Zer}}
\newcommand{\bC}{\mathbf{C}}
\newcommand{\bN}{\mathbf{N}}
\newcommand{\bZ}{\mathbf{Z}}
\newcommand{\bR}{\mathbf{R}}
\newcommand{\bQ}{\mathbf{Q}}
\newcommand{\bK}{\mathbf{K}}
\newcommand{\bD}{\mathbf{D}}

\newcommand{\scalar}[2]{\langle #1,#2\rangle}
\newcommand{\pairing}[2]{\langle #1,#2\rangle}

\newcommand{\Teis}[2]{
   \setlength{\unitlength}{1ex}
   \begin{picture}(2,0)(0,0.4)
      \put(0,1.1){\line(1,0){2}}
      \put(0,0.9){\line(1,0){2}}
      \put(1,1.2){\makebox(0,0)[b]{$\scriptstyle #1$}}
      \put(1,0.8){\makebox(0,0)[t]{$\scriptstyle #2$}}
   \end{picture}}

\newcommand{\Teisssr}[4]{
   \setlength{\unitlength}{1ex}
   \begin{picture}(#3,3)(0,0.4)
      \put(0,1.15){\line(1,0){#3}}
      \put(0,0.85){\line(1,0){#3}}
      \put(#4,1.3){\makebox(0,0)[b]{$#1$}}
      \put(#4,0.7){\makebox(0,0)[t]{$#2$}}
   \end{picture}}

\newtheorem{Theorem}{Theorem}[section]
\newtheorem{Lemma}[Theorem]{Lemma}
\newtheorem{Corollary}[Theorem]{Corollary}
\newtheorem{Remark}[Theorem]{Remark}
\newtheorem{Property}[Theorem]{Property}

\newtheorem{Definition}[Theorem]{Definition}
\newtheorem{Example}[Theorem]{Example}
\newenvironment{proof}[1][Proof]{\textbf{#1.} }{\
\rule{0.5em}{0.5em}}

\title{Quasi-ordinary singularities: tree model, discriminant and irreducibility
\footnotetext{
     \noindent   \begin{minipage}[t]{4.2in}
       {\small
       2000 {\it Mathematics Subject Classification:\/} Primary 32S55;
       Secondary 14H20.\\
       Key words and phrases: quasi-ordinary singularity,
       Newton polytope, discriminant, tree model, irreducibility.\\
       The first-named author was partially supported by the Spanish
       Projects PNMTM 2007-64007 and MTM2012-36917-C03-01.}
       \end{minipage}}}

\author{Evelia R.\ Garc\'{\i}a Barroso and Janusz Gwo\'zdziewicz}

\begin{document}
\maketitle

\begin{abstract}
\noindent Let $f(Y)\in\bK[[X_1,\dots,X_d]][Y]$ be a quasi-ordinary Weierstrass polynomial with coefficients in 
the ring of formal power series over an algebraically closed field of characteristic zero. 
In this paper we study the discriminant $D_f$ of $f(Y)-V$, where $V$ is a new variable. 
We show that the Newton polytope of $D_f$ depends only on contacts between the roots of $f(Y)$.
Then we prove  that  $f(Y)$ is irreducible if and only if 
the Newton polytope of $D_f$ satisfies some arithmetic conditions. 
Finally we generalize these results to quasi-ordinary power series. 
\end{abstract}

\section{Introduction}

\noindent Classically the irreducibility of singular plane curves was studied by resolving the singularity or using  approximate roots (Abhyankar criterion). More recently, in \cite{GB-Gwo} and \cite{Kodai} we use discriminants and the so called {\em Jacobian Newton polygon} introduced by Teissier in \cite{Teissier}. In \cite{Assi} the author gives an irreducibility criterion for {\em quasi-ordinary polynomials}  that generalizes the approach of Abhyankar for plane curves. In \cite{Manuel}  Gonz\'alex Villa characterizes the irreducible quasi-ordinary polynomials in terms of its {\em Newton process} (a way to encode  the resolution).  Previously, in \cite{GB-GP} (Theorem 3) the authors proved that if a power series is irreducible and has a {\em polygonal Newton polytope} (the maximal dimension of its compact faces equals one), then it has  only one compact edge,
which generalizes the case of plane curve germs.

\medskip

\noindent In this note, we study the irreducibility of a quasi-ordinary Weierstrass polynomial
$f(Y)\in\bK[[\underline{X}]][Y]$ from the point of view \cite{GB-Gwo} and \cite{Kodai}. 
We consider the Newton polytope $\Delta(D_f)$ of the discriminant $D_f(\underline{X},V)=\mbox{Discr}_Y(f(Y)-V)$, where $V$ is a new variable.

\medskip 

\noindent 
The main result of the article is Theorem~\ref{Th:6} which states that 
if $p(Y)$, $f(Y)$ are quasi-ordinary Weierstrass polynomials 
such that $\Delta(D_{p})=\Delta(D_{f})$ and $f$~is irreducible 
then $p$ is also irreducible. 

\medskip

\noindent Our tool is the {\em tree model} associated with a quasi-ordinary  polynomial, also called {\em Kuo-Lu} tree.
This combinatorial object is a natural generalization 
of a tree introduced in \cite{Kuo-Lu}.
The tree model $T(f)$ of a polynomial $f(Y)$ 
depends only on contacts between the roots of $f(Y)$.  

\medskip

\noindent In Theorem~\ref{Th:2} we give an explicit formula 
expressing the Newton diagram $\Delta(D_f)$ by $T(f)$.  
Then, after some preparatory work, we characterize in Theorems~\ref{Th:4} and \ref{Th:5}
the tree models of irreducible quasi-ordinary Weierstrass polynomials. 
These are tree models with the highest possible level of symmetry. 

\medskip 
\noindent The proof of Theorem~\ref{Th:6} is based on above results and its idea is to show that if 
$\Delta(D_{p})=\Delta(D_{f})$ and the tree model $T(f)$ has a high level of symmetry, then 
$T(p)$ has the same structure as $T(f)$. 

\medskip

\noindent A consequence of the main result is Theorem~\ref{Theo1} 
which presents an arithmetical test of irreducibility for quasi-ordinary Weierstrass polynomials. 
As an illustration we apply this test to three examples of quasi-ordinary polynomials from \cite{Assi}.

\medskip

\noindent Finally in Section \ref{qo-series} we generalize the notion of the discriminant $D_f(\underline{X},V)$, 
which was previously defined for quasi-ordinary Weierstrass polynomials, to  $Y$-{\em regular quasi-ordinary power series}  and we generalize the criterion of irreducibility to such power series.

\section{Quasi-ordinary Weierstrass polynomials}

\noindent While the term {\em quasi-ordinary} appears in the 60s with Zariski paper \cite{Zariski} and Lipman thesis (\cite{Lipman-tesis}), the study of these objects goes back at least to the paper \cite{Jung} of 
Jung. In this section we recall the notion of {\em quasi-ordinary Weierstrass polynomials}  and some results that will be useful in the development of this note.

\medskip

\noindent Let $\bK$ be an algebraically closed field of characteristic zero and let 
\begin{equation}\label{Eq:1}
f(Y)=Y^n+a_1(X_1,\dots,X_d)Y^{n-1}+\cdots+a_n(X_1,\dots,X_d)\in \bK[[\underline{X}]][Y]  
\end{equation}
be a unitary polynomial with coefficients in the ring of formal power series 
in $\underline{X}=(X_1,\dots,X_d)$.
Such a polynomial is called \emph{quasi-ordinary} if its $Y$-discriminant 
equals 
$X_1^{\alpha_1}\cdots X_d^{\alpha_d}u(\underline{X})$, where $\alpha_i\in\bN$ and $u(\underline{X})$ is a unity in  $\bK[[\underline{X}]]$, that is $u(0)\neq0$.
We call~$f(Y)$ a \emph{Weierstrass polynomial} if $a_i(0)=0$ for all $i=1,\dots,n$.

\begin{Theorem}[Abhyankar-Jung Theorem \cite{Jung}, \cite{P-R}] 
Let $f(Y)\in\bK[[\underline{X}]][Y]$ be a quasi-ordinary Weierstrass polynomial. 
Then there is $k\in\bN\setminus\{0\}$ such that $f(Y)$ has its roots 
in~$\bK[[X_1^\frac{1}{k},\dots,X_d^\frac{1}{k}]]$.
\end{Theorem}

\noindent For every $d$-tuple $\alpha=(\alpha_1,\dots,\alpha_d)\in\bQ_{\geq0}^d$ denote 
$\underline{X}^{\alpha}=X_1^{\alpha_1}\cdots X_d^{\alpha_d}$.\\ 
\noindent Let $\Zer f=\{Y_1(\underline{X}),\dots, Y_n(\underline{X})\}$ be the set of roots of $f(Y)$
in $\bK[[X_1^\frac{1}{k},\dots,X_d^\frac{1}{k}]]$.

\noindent As the differences of roots divide the discriminant, we have for $i\neq j$

\[
Y_{ij}(\underline{X}):=Y_i(\underline{X})-Y_j(\underline{X})=\underline{X}^{\lambda_{ij}}u_{ij}(\underline{X}),\;\;\;
\hbox{\rm for some  }\lambda_{ij}\in(1/k)\bN^d,  u_{ij}(0)\neq0.
\]

\medskip

\noindent In the next we will write $Y_j$ instead of $Y_j(\underline{X})$ and $Y_{ij}$ instead of $Y_{ij}(\underline{X})$. We call $O\big(Y_i,Y_j\big):=\lambda_{ij}$ \emph{the contact between $Y_i$ and $Y_j$}. We put $O\big(Y_i,Y_i\big)=+\infty$.
\medskip

\noindent Let us introduce a partial order in $\bQ^d$: $(\alpha_1,\ldots,\alpha_d) \leq (\beta_1,\ldots,\beta_d) $ if and only if  $\alpha_i\leq \beta_i$ for all $i=1,\ldots,d$. 
Let us put by convention $\alpha<+\infty$ for $\alpha\in\bQ^d$.

\medskip
\begin{Lemma} [\cite{B-M}, Lemma 4.7]
\label{LE1}
Let $\alpha$, $\beta$, $\gamma\in\bN^d$ and let $a(\underline{X})$, $b(\underline{X})$, $c(\underline{X})$ be invertible elements 
of $\bK[[\underline{X}]]$. If 
$$ a(\underline{X})\underline{X}^{\alpha}-b(\underline{X})\underline{X}^{\beta}=c(\underline{X})\underline{X}^{\gamma}, $$ 
then either $\alpha\leq \beta$ or $\beta \leq \alpha$.
\end{Lemma}

\noindent Applying Lemma \ref{LE1} to $Y_{ik}$, $Y_{jk}$ and $Y_{ij}$ we see that for every $Y_i,Y_j,Y_k\in \Zer f$ one has $O(Y_i,Y_k)\leq O(Y_j,Y_k)$ or $O(Y_i,Y_k)\geq O(Y_j,Y_k)$.

\noindent Moreover, we have the {\em strong triangular inequality:} 

\[
O(Y_i,Y_j)\geq \min \{O(Y_i,Y_k), O(Y_j,Y_k)\}.
\]

\noindent Consequently for every subset $A\subset\Zer f$ the set of contacts 
between elements of~$A$ has 
the smallest element. 

\section{The tree model $T(f)$}
\label{tree-model}
In this section we construct the tree model $T(f)$ which encodes the contacts between the roots of $f(Y)$. 
Given  $h\in\bQ_{\geq0}^d$ we write  $Y_i \equiv Y_j\bmod h^{+}$  if $O(Y_i,Y_j)>h$.\\

\noindent Let $B=\Zer f$ and let $h(B)$ be the minimal contact between the elements of~$B$. 
We represent $B$ as a horizontal bar and call $h(B)$ the \emph{height} of $B$. 
The equi\-valence relation $\equiv \bmod  \,h(B)^{+}$  divides $B$ into equivalence classes $B_1$, \dots, $B_k$. 
 From $B$ we draw $k$ vertical segments and at the end of the $i$-th segment we place
a horizontal bar representing $B_i$. The bar $B_i$ is called a {\em postbar} of $B$ and we write $B\perp B_i$.  
For each $B_i$ we repeat this construction recursively. We do not draw the bars of infinite height. 

\medskip

\noindent Remark that for every bar $\bar B\in T(f)$ there exists a unique sequence  $B\perp B'\perp B''\perp \cdots \perp \bar B$ starting from the bar $B$ of the minimal height. 

\medskip

\noindent Let $\# A$ denotes the number of elements of the set $A$. 

\begin{Definition}
\label{page2}
To every bar $B\in T(f)$ we associate a $d$-tuple $q(B)\in\bQ_{\geq0}^d$ in the next way:
\begin{enumerate}
\item [(i)]  If $B$ is the bar of the lowest height then $q(B)= \#B \cdot h(B)$. 
\item [(ii)]  If $B\perp B'$ then $q(B')=q(B)+\#B'(h(B')-h(B))$.
\item [(iii)] If $h(B)$ is infinite then $q(B)$ is also infinite.
\end{enumerate}
\end{Definition}

\begin{Remark}
For $d=1$,  $q(B)$ becomes a rational number. In \cite{GB-Gwo} and \cite{GB-Gwo-L} this number was defined by using the order of certain substitutions. However Lemma 2.7 in \cite{GB-Gwo-L} states that $q(B)$ satisfies the recursive formula of Definition \ref{page2}. Hence, that definition coincides with the present one.
\end{Remark}

\medskip
\noindent 
Let $q=(q_1,\dots,q_d)\in\bQ_{\geq0}^d$ and let 
$k$ be a positive integer. We define the {\em elementary Newton polytope} 
$$ 
\Bigl\{\Teisssr{q}{k}{3}{1.5}\Bigr\}:=
\mbox{Convex Hull}\;\bigl( \{\, (q_1,\dots,q_d,0), (0,\dots,0,k)\,\}
+\bR_{\geq0}^{d+1}\bigr)\;.
$$

\begin{Example} The elementary Newton polytope $\left\{\Teisssr{(2,1)}{4}{4}{2}\right\}$ is
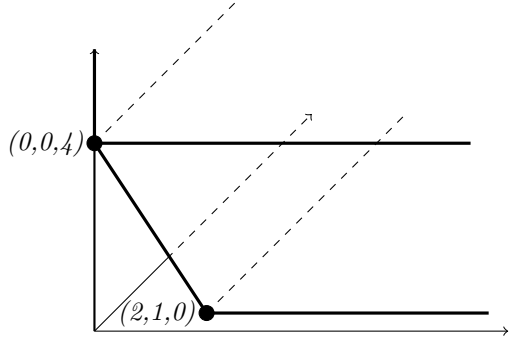
\begin{center}
\begin{tikzpicture}[scale=2.5]
\draw [->](0,0,0) -- (2.2,0,0); \draw[->](0,0,0) -- (0,1.5,0) ; \draw(0,0,0) -- (0,0,-1);
\draw[thick] (0,0,0) -- (0,1,0) node[left] {(0,0,4)};
\draw[dashed](0,1,0) -- (0,1,-2);
\draw[very thick](0,1,0)-- (2,1,0);
\draw[very thick](0,1,0)-- (0,1.5,0);
\draw[very thick] (0,1,0) -- (0.5,0,-0.25) node[left] {(2,1,0)};
\draw[very thick](0.5,0,-0.25) -- (2,0,-0.25);
\draw[dashed](0.5,0,-0.25) -- (0.5,0,-3);
\draw[->][dashed](0,0,-1) -- (0,0,-3);
\node[draw,circle,inner sep=2pt,fill=black] at (0,1,0) {};
\node[draw,circle,inner sep=2pt,fill=black] at (0.5,0,-0.25) {};
\end{tikzpicture}
\end{center}
\end{Example}

\noindent With each tree model $T$ we associate the Newton polytope 
\begin{equation}
\label{pol}
\Delta_T= \sum_{B\in \widetilde T} \left\{\Teisssr{(t(B)-1)q(B)}{t(B)-1}{16}{8}\right\},
\end{equation}

\noindent where $\widetilde T$ is the set of bars 
 $B\in T$ of finite height, $t(B)$ is the number of postbars of $B$ and the sum denotes the Minkowski sum (see \cite{Ewald}, Chapter 4, Definition 1.1).

\section{Newton polytope of the discriminant}

\noindent Let $h(\underline{X})=\sum_{\underline{i}}a_{\underline{i}}\underline{X}^{\underline{i}}$ be a power series in $s$ variables and coefficients in $\bK$. The {\em Newton polytope} $\Delta(h)$ of $h$ is the convex hull of the set
$\displaystyle \bigcup_{a_{\underline{i}}\neq 0} \{ \underline{i}+\bR_{\geq0}^{s}\}$. In two variables case the Newton polytope is called the {\em Newton diagram}.

\medskip

\noindent If $g(Y)\in \bC\{X_1\}[Y]$ is a Weierstrass polynomial then the Newton diagram of the $Y$-discriminant  of 
$g(Y)-V$, where $V$ is a new variable, is determined by the tree model of $g$ (see Lemma \ref{L:3}).  In this section we generalize this result to quasi-ordinary Weierstrass polynomials in $\bK[[X_1,\ldots,X_d]][Y]$.

\begin{Theorem}\label{Th:2}
 Let $f(Y)\in\bK[[\underline{X}]][Y]$ be a quasi-ordinary Weierstrass polynomial 
and let $D_f(\underline{X},V)$ be the $Y$-discriminant 
of the polynomial $f(Y)-V$, where $V$ is a new variable.
Then $\Delta(D_f)=\Delta_{T(f)}$.
\end{Theorem}

\medskip

\noindent  Set $f(Y)=Y^n+a_1(X_1,\dots,X_d)Y^{n-1}+\cdots+a_n(X_1,\dots,X_d)\in \bK[[\underline{X}]][Y]$. Let  $X_1=T^{c_1}$, \dots , $X_d=T^{c_d}$ be  monomial substitutions, where $T$ is a new variable and $c_i$ are positive integers.
Set $\underline{c}=(c_1,\dots,c_d)$ and let 
\begin{equation}\label{Eq:2} 
g(Y)=Y^n+a_1(T^{c_1},\dots,T^{c_d})Y^{n-1}+\cdots+a_n(T^{c_1},\dots,T^{c_d})\in \bK[[T]][Y]. 
\end{equation}

\noindent Remark that if $c_i\geq n$ for $i=1,\dots, d$ 
then the order of $a_i(T^{c_1},\dots,T^{c_d})$ is bigger than 
or equal to $n$ for $i=1,\dots, n$. 
In particular the initial form of $g$, treated as a power series in variables 
$T$ and $Y$, is not divisible by $T$ since $Y^n$ is one of its terms.

\begin{Lemma}\label{L:2}
There is  a bijective correspondence between the bars of $T(f)$ and the bars of $T(g)$. 
Moreover, if $B$ and $\bar B$ are corresponding bars of $T(f)$ and $T(g)$ respectively then 
$h(\bar B) = \scalar{\underline{c}}{h(B)}$ and  $q(\bar B) = \scalar{\underline{c}}{q(B)}$,
where $\scalar{\cdot}{\cdot}$ denotes the standard scalar product. 
\end{Lemma}

\noindent \begin{proof}
Set $T^{\underline{c}}=(T^{c_1},\dots,T^{c_d})$. 
Clearly $\Zer g=\{Y_1(T^{\underline{c}}),\dots,Y_n(T^{\underline{c}})\}$
and $O(Y_i(T^{\underline{c}}),Y_j(T^{\underline{c}}))=\scalar{\underline{c}}
{O(Y_i(\underline{X}),Y_j(\underline{X}))}$ for $i\neq j$.
Hence every bar $B=\{Y_{i_1}(\underline{X}),\dots,Y_{i_k}(\underline{X})\}$ of $T(f)$ yields the bar 
$\bar B=\{Y_{i_1}(T^{\underline{c}}),\dots,Y_{i_k}(T^{\underline{c}})\}$ of $T(g)$ of height 
$\scalar{\underline{c}}{h(B)}$. 
Taking the scalar product by $\underline{c}$ of the equations appearing in  Definition \ref{page2}  we get the second part of the lemma. 
\end{proof}

\medskip
\noindent Further, in this section,  we write $(\underline{x},y)$ for $(x_1,\dots,x_d,y)\in\bR^{d+1}$.

\medskip

\begin{Corollary}\label{C:1}
Let  ${\pi}:\bR^{d+1}\to\bR^2$  be the linear mapping given by 
$(\underline{x},x_{d+1})\mapsto (\scalar{\underline{c}}{\underline{x}},x_{d+1})$. 
Then $\pi(\Delta_{T(f)})=\Delta_{T(g)}$.
\end{Corollary}

\noindent \begin{proof}
Corollary \ref{C:1} follows from Lemma~\ref{L:2} and two easy observations: \\
$\pi\Bigl(\Bigl\{\Teisssr{q}{k}{2}{1}\Bigr\}\Bigr)=\left\{\Teisssr{\scalar{\underline{c}}{q}}{k}{5}{2.5}\right\}$
for every elementary Newton polytope  $\Bigl\{\Teisssr{q}{k}{2}{1}\Bigr\}\subset\bR_{\geq0}^{d+1}$,
\and 
$\pi(\Delta_1+\Delta_2)=\pi(\Delta_1)+\pi(\Delta_2)$ 
for all Newton polytopes $\Delta_1,\Delta_2\subset\bR_{\geq0}^{d+1}$.
\end{proof}

\medskip

\begin{Lemma}\label{L:3}
Let $D_g(T,V)$ be the $Y$-discriminant of the polynomial $g(Y)-V$, where $V$ is a new variable.
Assume that $T$ does not divide the initial form of $g$ treated as a power series in two variables.
Then $\Delta(D_g)=\Delta_{T(g)}$.
\end{Lemma}

\medskip

\noindent Lemma \ref{L:3} was proved in \cite{GB-Gwo} (see page 691)  for Weierstrass polynomials in $\bC\{T\}[Y]$. However its proof can be generalized without any problems to Weierstrass polynomials with coefficients  in the ring  $\bK[[T]]$. 

\medskip

\noindent \begin{proof}[Proof of Theorem \ref{Th:2}] 
For every Newton polytope $\Delta\subset\bR_{\geq0}^{k}$  and every  $v\in\bR_{\geq0}^{k}$ we define 
the {\em support function} $l(v,\Delta)=\min\{\scalar{v}{\alpha}:\alpha\in\Delta\}$. 
To prove the theorem it  is enough to show that the
support functions  $l(\cdot,\Delta(D_f))$ and $l(\cdot,\Delta_{T(f)})$ are equal. As these functions are continuous 
it suffices to show the equality on a dense subset of $\bR_{\geq0}^{d+1}$.

\medskip

\noindent Let $c=(c_1,\dots,c_{d+1})=(\underline{c},c_{d+1})\in\bR_{\geq0}^{d+1}$, where
$\underline{c}=(c_1,\dots,c_d)$.

\medskip

\noindent Perturbing $c$ a little we may assume that the hyperplane 
$\{\,\alpha\in\bR^{d+1}:\scalar{c}{\alpha}=l(c,\Delta(D_f)\,\}$ 
supports $\Delta(D_f)$ at exactly one point $\check{\alpha}=(\underline{\check{\alpha}},\check{\alpha}_{d+1})$. 
Since after a small change of $c$ the support point remains the same, we can assume,
perturbing  $c$ again if necessary, that all $c_i$ are positive rational numbers.

\medskip

\noindent We will show that 
\begin{equation}\label{Eq:3}
l(c,\Delta_{T(f)})=l(c,\Delta(D_f)).
\end{equation}

\noindent Multiplying $c$ by the common denominator of $c_1$, \dots, $c_{d+1}$ we may 
assume that all $c_i$ are integers bigger than or equal to $\deg f$. At this point of the proof 
we fixed $c$. Let $g(Y)$ be the Weierstrass polynomial given by~(\ref{Eq:2}).  
We claim that
$l(c,\Delta_{T(f)})=l\bigl((1,c_{d+1}),\Delta_{T(g)}\bigr)$ and 
$l(c,\Delta(D_f))=l\bigl((1,c_{d+1}),\Delta(D_g)\bigr)$.

\medskip
\noindent
First equality follows from Corollary~\ref{C:1} and the
identity $\scalar{c}{\alpha}=\scalar{(1,c_{d+1})}{\pi(\alpha)}$ for $\alpha\in\bR^{d+1}$.

\medskip\noindent 
Let  $D_f(\underline{X},V)=\sum_{\alpha}
d_{ \alpha}{\underline{X}}^{\underline{\alpha}}V^{\alpha_{d+1}}$, 
where $\alpha=(\underline \alpha, \alpha_{d+1})$. 
As the discriminant commutes with base change  we get by~(\ref{Eq:2})
$D_g(T,V)=\sum_{ \alpha}d_{\alpha}T^{\scalar{\underline{c}}{\underline{\alpha}}}V^{\alpha_{d+1}}$. 
Since the hyperplane 
$\{\,\alpha\in\bR^{d+1}:\scalar{c}{\alpha}=l(c,\Delta(D_f)\,\}$ 
supports $\Delta(D_f)$ at $\check{\alpha}$, the monomial 
$d_{\check{\alpha}}T^{\scalar{\underline{c}}{\underline{\check{\alpha}}}}V^{\check{\alpha}_{d+1}}$ satisfies the equality 
$\scalar{\underbar{c}}{\underline{\check{\alpha}}}+c_{d+1}\check{\alpha}_{d+1}=l(c,\Delta(D_f))$, 
while for all other monomials $d_{\alpha}T^{\scalar{\underline{c}}
{\underline{\alpha}}}V^{\alpha_{d+1}}$ with $d_{\alpha}\neq0$ appearing in the sum 
$\sum_{ \alpha}d_{\alpha}T^{\scalar{\underline{c}}{\underline{\alpha}}}V^{\alpha_{d+1}}$ 
we have  $\scalar{\underbar{c}}{\underline{\alpha}}+c_{d+1}\alpha_{d+1}>l(c,\Delta(D_f))$.
Hence 
$l\bigl((1,c_{d+1}),\Delta(D_g)\bigr)=
\scalar{\underbar{c}}{\underline{\check{\alpha}}}+c_{d+1}\check{\alpha}_{d+1}=
l(c,\Delta(D_f))$.

\medskip

\noindent 
 By Lemma~\ref{L:3} $\Delta_{T(g)}=\Delta(D_g)$ which together with the just proved 
claim gives~(\ref{Eq:3}). This completes the proof because $c$ is sufficiently general. 
\end{proof}

\medskip

\noindent From Theorem \ref{Th:2}, Corollary \ref{C:1} and Lemma \ref{L:3} we get $\pi(\Delta(D_f))=\pi(\Delta_{T(f)})=\Delta_{T(g)}=\Delta(D_g)$, which gives us

\begin{Corollary}\label{C:2}
Let  ${\pi}:\bR^{d+1}\to\bR^2$  be the linear mapping given by 
$(\underline{x},x_{d+1})\mapsto (\scalar{\underline{c}}{\underline{x}},x_{d+1})$. 
Then $\pi(\Delta(D_f))=\Delta(D_g)$, where $f$ and $g$ are quasi-ordinary Weierstrass polynomials given by the equations 
(\ref{Eq:1}) and  (\ref{Eq:2}) respectively.
\end{Corollary}

\section{Symmetry of the tree model}
\label{symmetry}
\noindent In this section we describe symmetries of the tree model associated with a quasi-ordinary Weierstrass polynomial $f(Y)$. 

\medskip

\noindent Let $U=\{\omega\in\bK: \omega^k=1\}$ be the multiplicative group of k-th roots of unity.
With every $d$-tuple $\underline{\epsilon}=(\epsilon_1,\dots,\epsilon_d)\in U^d$
we associate the $\bK$-algebra homomorphism 
$\phi_{\underline{\epsilon}}:\bK[[X_1^\frac{1}{k},\dots,X_d^\frac{1}{k}]]\to\bK[[X_1^\frac{1}{k},\dots,X_d^\frac{1}{k}]]$, such that $\phi_{\underline{\epsilon}}(X_i^{\frac{1}{k}})=\epsilon_i X_i^{\frac{1}{k}}$ for $i=1,\dots,d$.  
Since $\phi_{\underline{\epsilon}}(X_i)=\epsilon_i^k X_i=X_i$, the homomorphism $\phi_{\underline{\epsilon}}$ is the  identity on $\bK[[\underline{X}]]$. 
For every $\underline{\epsilon},\underline{\omega}\in U^d$ we have $\phi_{\underline{\epsilon}}\circ\phi_{\underline{\omega}}=\phi_{\underline{\epsilon}\cdot \underline{\omega}}$, where the product  $\underline{\epsilon}\cdot \underline{\omega}$ is componentwise.
Hence the {\em star operation} $\underline{\epsilon}*\psi(\underline{X}):=\phi_{\underline{\epsilon}}(\psi(\underline{X}))$ is an 
action of the group $U^d$ on $\bK[[X_1^\frac{1}{k},\dots,X_d^\frac{1}{k}]]$.

\medskip
\noindent If 
$\psi(\underline{X})=\sum_{\underline{\alpha}\in(1/k)\bN^d}c_{\underline{\alpha}}{\underline{X}}^{\underline{\alpha}}$
then
$\underline{\epsilon}*\psi(\underline{X})= \sum_{\underline{\alpha}\in(1/k)\bN^d}c_{\underline{\alpha}}
\underline{\epsilon}^{k\underline{\alpha}}{\underline{X}}^{\underline{\alpha}}$.

\medskip

\noindent We will show that the star operation
permutes the set $\Zer f$ and is transitive on $\Zer f$ providing $f(Y)$ is 
irreducible in $\bK[[\underline{X}]][Y]$. Moreover, it preserves the contact.  

\medskip

\noindent To be more precise, we have 

\begin{Property}\label{P:1} \ 

\begin{itemize}
\item[(i)] $\underline{\epsilon}*\Zer f=\Zer f$ for every $\underline{\epsilon}\in U^d$.
\item[(ii)] If $f(Y)$ is irreducible in $\bK[[\underline{X}]][Y]$ 
then $\Zer f=U^d*Y_i$ for every $Y_i\in\Zer f$. 
\item[(iii)] $O(Y_i,Y_j)=O(\underline{\epsilon}*Y_i,\underline{\epsilon}*Y_j)$ for every $\underline{\epsilon}\in U^d$ and  $i\neq j$. 
\end{itemize}
\end{Property}

\noindent \begin{proof}
Fix $\underline{\epsilon}\in U^d$.  The  homomorphism $\phi_{\underline{\epsilon}}$ naturally extends to the homomorphism 
$\Phi_{\underline{\epsilon}}:
\bK[[X_1^\frac{1}{k},\dots,X_d^\frac{1}{k}]][Y]\to\bK[[X_1^\frac{1}{k},\dots,X_d^\frac{1}{k}]][Y]$.
Acting by $\Phi_{\underline{\epsilon}}$ on $f(Y)=\prod_{i=1}^n[Y-Y_i]$ we get 
$f(Y)=\Phi_{\underline{\epsilon}}(f(Y))=\prod_{i=1}^n[Y-\phi_{\underline{\epsilon}}(Y_i)]$ which proves \emph{(i)}.

\medskip

\noindent Fix $Y_i\in\Zer f$ and let $f_1(Y)=\prod_{Y(\underline{X})\in U^d*Y_i}[Y-Y(\underline{X})]$. 
For every $\underline{\epsilon}\in U^d$ we  have $\Phi_{\underline{\epsilon}}(f_1(Y))=f_1(Y)$. 
Since the action of $U^d$ on $f_1(Y)$ is trivial the polynomial $f_1(Y)$ has  
coefficients in the ring $\bK[[\underline{X}]]$. 
By~\emph{(i)} all roots of $f_1(Y)$ are the roots of $f(Y)$. Assuming that $f(Y)$ is irreducible in $\bK[[\underline{X}]][Y]$, we get $f_1(Y)=f(Y)$ which proves \emph{(ii)}.

\medskip

\noindent Statement \emph{(iii)} follows directly from the definition of the {\em star action}. 
\end{proof}

\medskip
\noindent For every $\underline{\epsilon}\in U^d$ the mapping $\Zer f\ni Y_i\to\underline{\epsilon}*Y_i\in\Zer f$ preserves contacts. 
Let $B=\{Y_{i_1},\dots,Y_{i_s}\}$ be a bar of $T(f)$. Then 
$\underline{\epsilon}*B=\{\underline{\epsilon}*Y_{i_1},\dots,\underline{\epsilon}*Y_{i_s}\}$ is also a bar of  $T(f)$ of the same height. Thus $U^d\times T(f)\ni (\underline{\epsilon},B)\to \underline{\epsilon}*B\in T(f)$ is an action of the group $U^d$ on $T(f)$ which 
for each fixed $\underline{\epsilon}$ yields a symmetry of $T(f)$  preserving heights.  

\medskip
\noindent 
Every bar $\underline{\epsilon}*B$ will be called \emph{conjugate} to $B$.
Further in this section  we count the number of conjugates  of $B\in T(f)$. 
To this aim we employ the theory of {\em dual groups}. 

\medskip

\noindent Let $C$ be a cyclic group of order $k$ and let $G$ be a finite commutative group such that $kg=0$ for every $g\in G$. Recall that the {\em dual} of $G$, denoted $G^*$, is the group 
of homomorphisms from $G$ to $C$. 
 The main theorem of dual groups states that $G^*$ is isomorphic to $G$. 
\medskip

\noindent Let $A$, $A'$ be commutative groups. 
The mapping $A\times A'\to C$, $\;(x,x')\to \pairing{x}{x'}$ is called a {\em pairing} if for every 
$x'\in A'$ the mapping $\phi_{x'}=\pairing{\cdot}{x'}$ is a homomorphism of $A$ to $C$ and for every
$x\in A$ the mapping $\psi_x=\pairing{x}{\cdot}$ is a homomorphism of $A'$ to $C$.

\medskip
\noindent  For every $a\in A$ and $a'\in A'$ we introduce the orthogonal relation $a\perp a'$ if and only if 
$\pairing{a}{a'}$ is the identity element of $C$. For every set $B\subset A$ we denote by $B^{\perp}$ 
the set $\{x'\in A': \; b\perp x'\; \mbox{ for all }\; b\in B\}$. We make a similar definition of $(B')^{\perp}$ for $B'\subset A'$.

\medskip
\begin{Theorem}[\cite{Lang}, Theorem 9.2]
\label{th:Lang}
Let $A\times A' \to C$ be a pairing of two abelian groups into a finite cyclic group $C$. Assume that $A'$ is finite. Then $A'/A^{\perp}$ is isomorphic to the dual group of $A/(A')^{\perp}$.
\end{Theorem}

\begin{Corollary}
\label{orthogonal}
Let $A\times A' \to C$ be a pairing of two abelian groups into a finite cyclic group $C$. 
Assume that $A'$ is finite. 
If $M$, $N$  are subgroups of $A$ such that $A^{'\perp}\subset N\subset M$ then 
$[M:N]=[N^\perp:M^\perp]$.
\end{Corollary}

\noindent \begin{proof}
First, we will show that $(N^{\perp})^{\perp}=N$.

\noindent Let $a\in A\setminus N$. Then there exists $a'\in N^{\perp}$ such that $a\not\perp a'$. 
Indeed, if this is not the case then $N^{\perp}=N_1^{\perp}$, 
where $N_1$ is the group generated by $N\cup\{a\}$. 
By Theorem~\ref{th:Lang} the group $A'/N^{\perp}=A'/N_1^{\perp}$ would be dual 
of $N/(A')^{\perp}$ and of $N_1/(A')^{\perp}$ which is impossible because these groups have different number of elements since the coset of $a$ belongs to $N_1/(A')^{\perp}$ but not in $N/(A')^{\perp}$.  This shows that $a\notin (N^{\perp})^{\perp}$. Since $a$ is an arbitrary element of $A\setminus N$, we have $(N^{\perp})^{\perp}\subset N$. 

\noindent Let $a\in N$. Then for every $a'\in N^{\perp}$ we have $a\perp a'$. 
Consequently $a\in(N^{\perp})^{\perp}$ which gives $N\subset (N^{\perp})^{\perp}$. 
The first part of the proof is finished. 

\noindent It follows from Theorem~\ref{th:Lang} applied to  the pairing $M\times N^\perp\to C$ that 
$N^\perp/M^\perp$ is the dual of $M/(N^{\perp})^{\perp}=M/N$. Since a finite abelian group 
is isomorphic to its dual, we get $[M:N]=[N^\perp:M^\perp]$.
\end{proof}

\medskip

\noindent Let $B'$ be a postbar of $B\in T(f)$. Since all $Y_i,Y_j\in B'$  belong to the same equivalence class mod $h(B)^+$,  they have the same  term of exponent $h(B)$. Let $c$  be the coefficient of such a term. Following \cite{K-P} we write 
$B\perp_c B'$ and say that $B'$ {\em is supported at} $c$ on $B$. It is obvious that different postbars of $B$ are supported at different points.

\medskip

\begin{Definition}
\label{def}
 Let $B_0\perp_{c_0} B_1 \perp_{c_1} \cdots \perp_{c_{r-2}}B_{r-1}\perp_{c_{r-1}} B_r=  B$ be a sequence of bars of $T(f)$, where $B_0$ is the bar of the lowest height in $T(f)$. Let $H(B)=\{h(B_i)\;:\;c_i\neq 0,\;0\leq i\leq r-1\}=\{h_1,\dots,h_s\}$.
Then we call the lattice $N(B)=\bZ^d+\bZ h_1+\cdots+\bZ h_s$ the characteristic lattice of $B$. 
\end{Definition}

\noindent Note that if $Y(\underline{X})$ is any element of $B$ then $H(B)$ 
consist  of  such heights $h(B_i)$, $0\leq i \leq r-1$,  that
$\underline{X}^{h(B_i)}$ appears in $Y(\underline{X})$ with nonzero coefficient. 

\medskip 
\noindent Consider the pairing 
$(1/k)\bZ^d\times U^d\ni(\lambda,\underline{\epsilon})\to
\underline{\epsilon}^{k\lambda}\in U$. Directly from the definition it follows that for $\underline{\epsilon}\in U^d$ and $\lambda \in  (1/k)\bN^d$ holds $\underline{\epsilon}*\underline{X}^{\lambda}=\underline{X}^{\lambda}$ if and only if $ \lambda \perp \underline{\epsilon}$. It is easy to check that $(U^d)^{\perp}=\bZ^d$.

\begin{Theorem}\label{Th:3}
Every $B\in T(f)$
has $[N(B):\bZ^d]$ conjugates. \\ Let $B\perp_c B'$.
\begin{enumerate}
\item If $c\neq 0$ then there are $n(B)=[N(B)+\bZ h(B):N(B)]$ postbars of 
$B$ conjugate with $B'$. 

\item If $c=0$ 
then there there are no  postbars of $B$ conjugate with $B'$, expect $B'$ itself. 
\end{enumerate}
\end{Theorem}

\noindent \begin{proof}
Given $Y_i\in B$ and $\underline{\epsilon}\in U^d$ the contact between $Y_i$ and $\underline{\epsilon}*Y_i$ 
is bigger than or equal to $h(B)$ if and only if $h\perp \underline{\epsilon}$
for every $h\in H(B)$,
since otherwise the monomial $\underline{X}^h$ would appear in the difference $\underline{\epsilon}*Y_i-Y_i$ with nonzero coefficient. 
It follows that $\underline{\epsilon}*B=B$  if and only if $\underline{\epsilon} \in N(B)^{\perp}$.
Thus the stabilizer of $B$ under the action of $U^d$ is the group $N(B)^{\perp}$. 
By the orbit stabilizer theorem and Corollary~\ref{orthogonal} the set $U^d*B$ has
$[U^d:N(B)^{\perp}]=[N(B):\bZ^d]$ elements which proves the first part of the theorem.

\medskip

\noindent Let $B'$ be a postbar of $B$.
Then $\underline{\epsilon}*B'$ is a postbar of $B$ if and only if $\underline{\epsilon}*B=B$. 
Thus the set of postbars of $B$ which are conjugate to $B'$ is equal to $N(B)^{\perp}*B'$. 
By the just proven part of the theorem  $N(B')^{\perp}$ is the stabilizer of $B'$ under the action 
of $U^d$. 
By the orbit stabilizer theorem and Corollary~\ref{orthogonal}, the number of elements of 
$N(B)^{\perp}*B'$ equals $[N(B)^{\perp}:N(B')^{\perp}]=[N(B'):N(B)]$.

\medskip
\noindent Assume that  $c\neq 0$.
Then $[N(B'):N(B)]=[N(B)+\bZ h(B):N(B)]$. 

\medskip

\noindent Now, suppose that $c=0$. Since $N(B')=N(B)$, we get  $[N(B'):N(B)]=1$,  
hence the set of postbars of $B$ conjugate to $B'$ has one element.
\end{proof}

\begin{Corollary}\label{C:3}
If $B\in T(f)$ has $n(B)$ postbars then all of them are conjugate and they are supported at nonzero numbers.
\end{Corollary}

\section{The tree model of an irreducible polynomial}
\label{tree model irreducible}
Let $f(Y)\in\bK[[\underline{X}]][Y]$ be an irreducible quasi-ordinary Weierstrass polynomial. 
By Property~\ref{P:1} the action of $U^d$ on $\Zer f$ is transitive. This implies 
that  for fixed $Y_i$, the set of contacts $\{O(Y_j,Y_i):j\neq i\}$ does
not depend on the choice of $Y_i\in \Zer f$. 
If $\{O(Y_j,Y_i):j\neq i\}=\{h_1,\dots,h_g\}$, where $h_1<h_2<\dots<h_g$ then  $h_1,h_2,\dots,h_g$ is called the \emph{sequence of characteristic exponents of $f(Y)$}. The next lemma is  in \cite{Lipman}  (Remarks 5.8, page 469) but we give the proof for convenience of the reader.

\begin{Lemma}\label{lemma}
A finite sequence $h_1,h_2,\dots,h_g$ of elements from $\bQ_{\geq 0}^d$ is a sequence 
of characteristic exponents of an irreducible quasi-ordinary Weierstrass polynomial  
$f(Y)\in\bK[[X_1,\dots,X_d]][Y]$  if and only if
\begin{enumerate}
\item [(C1)]$h_1< h_2< \cdots < h_g$ and 
\item [(C2)] $h_i\not\in N_{i-1}:=\bZ^d+\bZ h_1+\cdots+\bZ h_{i-1}$ for $i=1,\dots, g$,
\end{enumerate}

\noindent  where $N_0=\bZ^d$.

\end{Lemma}

\noindent \begin{proof}
Let $f(Y)$ be an irreducible quasi-ordinary Weierstrass polynomial.
 Without loss of generality we may assume that all roots of $f(Y)$ 
belong to $\bK[[X_1^\frac{1}{k},\dots,X_d^\frac{1}{k}]]$. Let 
$Y_1$ be a fixed root of $f(Y)$  
and let $h_1,h_2,\dots,h_g$ be the sequence of its characteristic exponents. 
All roots of $f(Y)$ are conjugate by the action of $U^d$. Hence, by the definition 
of a sequence of characteristic exponents, for every 
$i\in \{1,\dots,g\}$ there exists $\underline\epsilon_i\in U^d$ such that
$h_i=O(\underline\epsilon_i*Y_1,Y_1)$.
This shows that all monomials $\underline X^{h_i}$ appear in $Y_1$ with non-zero coefficients. 
Moreover $\underline\epsilon_i*\underline X^{h_j}=\underline  X^{h_j}$ for $1\leq j<i$ and $\underline\epsilon_i*\underline X^{h_i}\neq \underline X^{h_i}$. 
We get $\underline\epsilon_i\in (N_{i-1})^\perp$ and
$\underline \epsilon_i^{kh_i}\neq 1$, hence $h_i\notin  N_{i-1}$ for $i=1,\dots, g$.

\medskip 

\noindent 
Now, assume that a sequence $h_1,h_2,\dots,h_g$ satisfies conditions $(C1)$ and $(C2)$.
Let $Y_1:=\underline X^{h_1}+\cdots+\underline X^{h_g}$. 
Clearly $Y_1\in\bK[[X_1^\frac{1}{k},\dots,X_d^\frac{1}{k}]]$ for some $k>0$.
Let $\{Y_1,\dots, Y_n\}$ be the set of conjugates of $Y_1$ by the action of $U^d$, 
where $U=\{\omega\in\bK: \omega^k=1\}$. 
Consider a polynomial $f(X)=\prod_{i=1}^n(Y- Y_i)$. As in the proof of Property~\ref{P:1} we show 
that $f(Y)$ is a polynomial with coefficients in the ring $\bK[[\underline X]]$ and that is irreducible 
over this ring. 

\medskip

\noindent Condition (C1) implies that the difference of any two roots of $f(Y)$ 
has a form $w(\underline X)\underline{X}^{h_l}$, where $w(0)\neq 0$ and $1\leq l\leq g$. 
Thus the discriminant of $f(Y)$, being the product of differences of the roots, equals 
$X_1^{\alpha_1}\cdots X_d^{\alpha_d}u(\underline{X})$, where $\alpha_i\in\bN$ and $u(\underline{X})$ is a unity in  $\bK[[\underline{X}]]$. This shows that $f(Y)$ is 
quasi-ordinary.

\medskip

\noindent By Condition (C2) we get $N_0\subsetneq N_1 \subsetneq \cdots \subsetneq N_g$ and consequently $U^d=N_0^{\perp}\supsetneq N_1^{\perp} \supsetneq \cdots \supsetneq N_g^{\perp}$. 
Take $\underline\epsilon\in U^d$. 
If  $\underline\epsilon\in N_{i-1}^{\perp}\backslash N_{i}^{\perp}$ 
then $O(\underline\epsilon*Y_1,Y_1)=h_i$ and if
$\underline\epsilon\in N_g^{\perp}$ then $\underline\epsilon*Y_1=Y_1$.  
Thus $h_1,h_2,\dots,h_g$  is the sequence of characteristic exponents of $f(Y)$.
\end{proof}

\medskip
\noindent Now we show that the tree model of an irreducible quasi-ordinary Weierstrass polynomial $f(Y)$ depends only on its sequence of characteristic exponents.

\begin{Theorem}\label{Th:4}
Let $f(Y)\in\bK[[\underline{X}]][Y]$ be an irreducible quasi-ordinary Weierstrass polynomial and let 
$h_1,h_2,\dots,h_g$ be the sequence of its characteristic exponents. 
Let $N_0=\bZ^d$ and $N_i=\bZ^d+\bZ h_1+\cdots+\bZ h_i$ for $i=1,\dots,g$.
Then the tree model $T(f)$ is characterized by two properties:
\begin{itemize}
\item[(i)] the set of the heights of bars of  $T(f)$ is $\{h_1,\dots,h_g,h_{g+1}\}$, 
               where $h_{g+1}=\infty$,
\item[(ii)] every bar of height $h_i$ has $[N_i:N_{i-1}]$ postbars 
               and all of them have the height $h_{i+1}$  for $i=1,\dots,g$. 
\end{itemize}
\end{Theorem}

\noindent \begin{proof}
Part (i) follows directly from the definition of the sequence of characteristic exponents.
Moreover, since the action of $U^d$ on $\Zer f$ is transitive, every bar of   height $h_i$ has 
only postbars of   height $h_{i+1}$, for $i=1,\dots,g$  and all bars of a fixed height are conjugate. 

\medskip

\noindent Let $B\in T(f)$. To prove part (ii) observe 
that if $h(B)=h_i$ then $N(B)=N_{i-1}$ since the monomials $\underline{X}^{h_j}$ for $1\leq i\leq g$
appear with nonzero coefficients in every $Y(\underline{X})\in \Zer f$. Applying part (ii) of Theorem~\ref{Th:3} 
to $B$ we see that $B$ has $[N_i:N_{i-1}]$ postbars conjugate with a given postbar $B'$ of $B$.
This completes the proof. 
\end{proof}

\medskip

\noindent A tree model $T$ satisfying conditions (i), (ii) of Theorem~\ref{Th:4} will be called the
 \emph{tree of type $(h_1,h_2,\dots,h_g)$.}

\medskip

\begin{Theorem}\label{Th:5}
If the tree model of a quasi-ordinary Weierstrass polynomial $f(Y)\in\bK[[\underline{X}]][Y]$  is  of 
type $(h_1,h_2,\dots,h_g)$ then $f(Y)$ is irreducible and $ h_1,h_2,\dots,h_g$ is the sequence 
of its characteristic exponents. 
\end{Theorem}

\noindent \begin{proof}
\noindent By conditions (i) and (ii) 
the tree $T(f)$ has $[N_g:N_{g-1}]\cdot[N_{g-1}:N_{g-2}]\cdots [N_1:N_0]=[N_g:\bZ^d]$ bars of infinite height. 

\medskip

\noindent 
The bar $B$ of $T(f)$ of the lowest height $h(B)=h_1$ hast at least two postbars. Let us choose one of them, $B'$,  which is supported at a nonzero number. 
Taking a similar choice of a postbar of $B'$ and continuing this procedure  $g-1$-times we arrive at a bar $\bar{B}$  of infinite height. It is clear that $N(\bar B)=N_g$. By Theorem \ref{Th:3} the
number of conjugates of $\bar B$ equals $[N_g:\bZ^d]$.

\medskip

\noindent Thus all bars of infinite height are conjugate. It follows that all  the roots of $f(Y)$ are conjugate by the action of $U^d$. Thus $f(Y)$ is irreducible in $\bK[[\underline{X}]][Y]$.
\end{proof}

\section{Irreducibility criterion}

\noindent In this section we consider two Weierstrass polynomials $p(Y)$ and $f(Y)$ such that $\Delta(D_{p})=\Delta(D_{f})$. We prove that $p(Y)$ is an irreducible quasi-ordinary polynomial if and only if $f(Y)$ is also. 

\begin{Theorem}\label{Th:6}
Let $f(Y)$, $p(Y)\in\bK[[\underline{X}]][Y]$ be quasi-ordinary Weierstrass polynomials 
such that $\Delta(D_f)=\Delta(D_{p})$.
Assume that $f(Y)$ is irreducible. Then $p(Y)$ is irreducible
and the sequences of characteristic exponents of $f(Y)$ and $p(Y)$ are equal. 
\end{Theorem}

\noindent \begin{proof}
Let $h_1,\dots,h_g$ be the sequence of characteristic exponents of $f(Y)$. 
By Theorem~\ref{Th:4} the tree model $T(f)$ is of  type $(h_1,\dots,h_g)$.
By Theorem~\ref{Th:5}  it is enough to show that $T(p)$ is also a tree  of  type $(h_1,\dots,h_g)$.

\medskip

\noindent First we will show that  the polynomials $f(Y)$ and $p(Y)$ have the same degree. 
If $f(Y)=Y^n+a_1Y^{n-1}+\cdots+ a_n$ then 
$\mbox{Discr}_Y(f(Y)-V)=d_0V^{n-1}+d_1V^{n-2}+\cdots+d_{n-1}$, where 
$d_0=(-1)^{(n+2)(n-1)/2}n^n$ (see \cite{P}, Lemma 2.1). It follows that $(\underline{0},\deg_Y f(Y)-1)$ is the point of the intersection of $\Delta(D_f)$ with the vertical axis having the smallest last coordinate. 
Thus 
the equality of the Newton polytopes $\Delta(D_f)$ and $\Delta(D_p)$ 
gives
$\deg f(Y)=\deg p(Y)$. 

\medskip

\noindent Now, let us compute recursively  the $d$-tuples $q(B)$ for $B\in T(f)$. 
Under the notations of Theorem~\ref{Th:4} we set $n_0=1$ and $n_i$=$[N_i:N_{i-1}]$ for $i=1,\dots,g$. 
By the symmetry of $T(f)$ every bar $B$ of height $h_i$, where $1\leq i\leq g$, has $n_0\cdots n_{i-1}$ conjugates. 
Moreover, by Definition \ref{page2} $q(B)$ is constant on the bars of the same height; we denote $q_i:=q(B)$ for such $B\in T(f)$ that
$h(B)=h_i$. We have 
\begin{equation}\label{Eq:R}
\begin{array}{lll}
q_1=n_1\cdots n_g h_1, & & \\
q_i=q_{i-1}+n_i\cdots n_g(h_i-h_{i-1}) & & \mbox{ for $i=2,\dots, g$.}
\end{array}
\end{equation}
\medskip

\noindent Hence 
\begin{equation}\label{Eq:6}
\Delta_{T(f)}= 
\sum_{i=1}^g \left\{\Teisssr{n_0\cdots n_{i-1}(n_i-1)q_i}{n_0\cdots n_{i-1}(n_i-1)}{20}{10}\right\}. 
\end{equation}

\noindent By (\ref{pol}) 
\begin{equation}\label{Eq:7}
\Delta_{T(p)}=\sum_{B\in \widetilde{T}(p)}\left\{\Teisssr{(t(B)-1)q(B)}{t(B)-1}{16}{8}\right\} .
\end{equation} 

\noindent Using the assumption $\Delta(D_f)=\Delta(D_{p})$ and Theorem~\ref{Th:2} we see that  polytopes given by (\ref{Eq:6}) and (\ref{Eq:7}) are equal. Hence
$\{\, q(B):B\in T(p)\,\}=\{q_1,\dots,q_g\}\cup\{\infty\}$. 

\medskip

\noindent Let $H_i=\{\,B\in T(p): q(B)=q_i\,\}$ for $i=1,\dots,g$. 
We will show, by induction on  $i$,  that
the set $H_i$ has $n_0\cdots n_{i-1}$ elements, the elements of $H_i$ are conjugate 
and  form a partition of $\Zer p$. 
Moreover, for every $B\in H_i$ we have
$h(B)=h_i$, $N(B)=N_{i-1}$ and $B$ has $n_i$ postbars which are conjugate.

\medskip
\noindent Let $B_0=\Zer p$ be the bar of the tree model $T(p)$ of the minimal height. 
Clearly $q(B_0)=q_1$ and $H_1=\{B_0\}$.
Since $B_0$ has $\deg p(Y)=\deg f(Y)=n_1\cdots n_g$ elements 
we get from (\ref{Eq:R}) and the formula for $q(B_0)$ (see Definition \ref{page2}) the equality $h(B_0)=h_1$. 

\medskip

\noindent Since $\Delta_{T(f)}=\Delta_{T(p)}$ we get from (\ref{Eq:6}) and (\ref{Eq:7}) the equality
$$ 
\left\{\Teisssr{(t(B_0)-1)q(B_0)}{t(B_0)-1}{16}{8}\right\} = 
\left\{\Teisssr{(n_1-1)q_1}{n_1-1}{10}{5}\right\}.
$$

\noindent Hence $B_0$ has $n_1$ postbars. Since $N(B_0)=\bZ^d$, 
we get $n(B_0)=[N(B_0)+\bZ h_1:N(B_0)]=[N_1:N_0]=n_1$ and 
by Corollary~\ref{C:3} all the postbars of $B_0$ are conjugate.

\medskip
\noindent Assume that the set $H_i$ has the desired properties.  We will prove them for $H_{i+1}$. 

\noindent Since $q(B)<q(B')$ for $B\perp B'$,  all the elements of  $H_{i+1}$ 
are postbars of the elements of $H_i$. By the inductive hypothesis all the
postbars  of the elements of $H_i$ are conjugate under the action of $U^d$. 
Hence all of them have the same height and  $H_{i+1}=\{\,B'\in T(p): B\perp B', B\in H_i\,\}$.  Since every $B\in H_i$ has $n_i$ postbars, by Corollary
\ref{C:3} every postbar $B'$ of $B$ is supported at  a nonzero number and $N(B')=N(B)+ \bZ h_i=N_i$.
The set $H_{i+1}$ has $n_0\cdots n_i$ elements, $H_{i+1}$  is a partition of $\Zer p$, 
 and every $B'\in H_{i+1}$ has $n_{i+1}\cdots n_g$ elements.

\medskip
\noindent Since the polytopes given in  (\ref{Eq:6}) and (\ref{Eq:7}) are equal, we get 
$$
\left\{\Teisssr{n_0\cdots n_{i}(n_{i+1}-1)q_{i+1}}{n_0\cdots n_{i}(n_{i+1}-1)}{22}{11}\right\}= 
\sum_{B\in H_{i+1}} \left\{\Teisssr{(t(B)-1)q(B)}{t(B)-1}{16}{8}\right\} = 
n_0\cdots n_{i}\left\{\Teisssr{(t(B')-1)q(B')}{t(B')-1}{16}{8}\right\} ,
$$
where $B'$ is a fixed element of $H_{i+1}$.  Consequently $B'$ has $n_{i+1}$ postbars.  

\medskip

\noindent
 By  Definition \ref{page2} we have $q(B')=q(B)+\#B'(h(B')-h(B))$ for $B\perp B'$.
 If $B\in H_i$ and  $B'\in H_{i+1}$, this gives us $q_{i+1}=q_i+(n_{i+1}\cdots n_g)(h(B')-h_i)$.
 Using formula~(\ref{Eq:R}) we get $h(B')=h_{i+1}$. 

\medskip

\noindent
Once we know the height $h(B')$ we also know  that $n(B')=[N_{i+1}:N_i]=n_{i+1}$. 
Hence by Corollary~\ref{C:3} $B'$ has $n_{i+1}$ postbars and all of them are conjugate.
\end{proof} 

\section{Arithmetical test of irreducibility}

\noindent In this section we consider Newton polytopes $\Delta=\sum_{i=1}^g \Bigl\{\Teis{L_i}{M_i}\Bigr\}\subset \bR_{\geq0}^{d+1}$,  where $\frac{1}{M_1}L_1<\frac{1}{M_2}L_2<\cdots < \frac{1}{M_g}L_g$. We associate to $\Delta$ the sequences:
\begin{enumerate}
\item $H_0=1$, $H_i=1+M_1+\dots+M_i$ for  $i\in \{1,\dots,g\}$,
\item $\gamma_i=\frac{H_{i-1}}{M_i}L_i$ for $i\in \{1,\dots,g\}$
\end{enumerate}

\noindent and the sequence of lattices $W_i=H_g\bZ^d+\bZ \gamma_1+\cdots+\bZ \gamma_i$, for $i\in \{0,\dots,g\}$.  We say that $\Delta$ is an $I$-{\em polytope} if and only if $[W_i:W_{i-1}]=H_i/H_{i-1}$ for $i\in \{1,\dots,g\}$.
Note that the $I$-polytopes  for $d=1$ are called {\em Merle polygons} in \cite{Kodai}.

\medskip

\noindent  The reader interested in computing the indices $[W_i:W_{i-1}]$, in an effective way, is encouraged to read Section (5.9) (page 469)  of \cite{Lipman}. For the convenience of the reader we prove this result in the appendix.

\medskip

\noindent Theorem \ref{Th:6} allows us to present  an arithmetical test of irreducibility for quasi-ordinary Weierstrass polynomials:

\begin{Theorem}\label{Theo1}
Let $f\in \bK[[X_1,\dots,X_d]][Y]$ be a Weierstrass polynomial. 
Then $f$ is irreducible and quasi-ordinary if and only if $\Delta(D_f)$ is an $I$-polytope.
\end{Theorem}

\noindent \begin{proof} 
Let $f$ be an irreducible quasi-ordinary Weierstrass polynomial
and let $h_1,\dots,h_g$ be the sequence of its characteristic exponents. 
By Lemma~\ref{lemma}  the numbers $n_i=[N_i:N_{i-1}]$, where 
$N_i=\bZ^d+\bZ h_1+\cdots+\bZ h_i$, are bigger than 1 for $i=1,\dots,g$. 
Consider an auxiliary sequence $\tilde\gamma_1,\dots, \tilde\gamma_g$ given by recurrence relations
\begin{equation}\label{Eq:R'}
\begin{array}{lll}
\tilde\gamma_1= h_1, & & \\
\tilde \gamma_i=n_{i-1}\tilde\gamma_{i-1}+h_i-h_{i-1} & & \mbox{ for $i=2,\dots, g$.}
\end{array}
\end{equation}
Let $n=n_1\cdots n_g$ and let $\gamma_i=n\tilde\gamma_i$.
Then it follows from ~(\ref{Eq:R}) and (\ref{Eq:6}) that 
\begin{equation}\label{Eq:6'}
\Delta_{T(f)}= 
\sum_{i=1}^g \left\{\Teisssr{(n_i-1)\gamma_i}{n_0\cdots n_{i-1}(n_i-1)}{20}{10}\right\}. 
\end{equation}

\noindent Let $L_i$ and $M_i$ denote the numerator and the denominator of the $i$-th term
of (\ref{Eq:6'}). It is easy to show by induction that 
$H_i:=1+M_1+\dots+M_i=n_1\cdots n_i$ for $i=1,\dots, g$. 
Hence $\gamma_i=(H_{i-1}/M_i)L_i$ for $i=1,\dots, g$.

\medskip

\noindent It follows from~(\ref{Eq:R'}) that 
$N_i=\bZ^d+\bZ h_1+\cdots+\bZ h_i=\bZ^d+\bZ\tilde\gamma_1+\cdots+\bZ\tilde\gamma_i$.
Since $H_g=n$ and $\gamma_i=n\tilde\gamma_i$ for $i=1,\dots,g$, we get 
$W_i=nN_i$  for $i=0,\dots,g$. This gives the arithmetic conditions 
$[W_i:W_{i-1}]=[N_i:N_{i-1}]=n_i=H_i/H_{i-1}$ for $i=1,\dots,g$.

\medskip

\noindent It remains to show that $\frac{1}{M_1}L_1<\frac{1}{M_2}L_2<\cdots < \frac{1}{M_g}L_g$.
Each inequality $(1/M_{i-1})L_{i-1}<(1/M_i)L_i$ 
can be written in equivalent form $n_{i-1}\gamma_{i-1}<\gamma_i$ which by~(\ref{Eq:R'}) is 
equivalent to $h_{i-1}<h_i$. Since characteristic exponents form an increasing sequence, this part 
of the proof is finished. 

\medskip

\noindent We proved that $\Delta_{T(f)}$, which is the Newton polytope of $D_f$,  
is an $I$-polytope.

\medskip

\noindent Now, assume that  $\Delta(D_f)=\sum_{i=1}^g \Bigl\{\Teis{L_i}{M_i}\Bigr\}$ is an $I$-polytope.  Let $n_i=H_i/H_{i-1}$ for $i=1,\dots g$. 
Then $n_i$ are integers bigger than 1 and $H_i=n_1\cdots n_i$ for $i=1,\dots, g$. 
We get $M_i=H_i-H_{i-1}=n_1\cdots n_{i-1}(n_i-1)$ 
and  $L_i=(M_i/H_{i-1})\gamma_i=(n_i-1)\gamma_i$
for $i=1,\dots, g$. 

\medskip

\noindent Let $n=n_1\cdots n_g$ and let $\tilde\gamma_i=(1/n)\gamma_i$ for $i=1,\dots, g$. 
This time we use the recurrence relations~(\ref{Eq:R'}) to define the sequence $h_1,\dots, h_g$.
As in the first part of the proof we can show that if $N_i=\bZ^d+\bZ h_1+\cdots+\bZ h_i$ then 
$W_i=nN_i$. This gives $[N_i:N_{i-1}]=[W_i:W_{i-1}]=n_i>1$ for $i=1,\dots,g$. Therefore 
$N_0\subsetneq N_1 \subsetneq \cdots \subsetneq N_g$.  

\medskip

\noindent  Again, as in the first part of the proof, we show that the inequalities 
$\frac{1}{M_1}L_1<\frac{1}{M_2}L_2<\cdots < \frac{1}{M_g}L_g$ are equivalent to
 the inequalities $h_1<h_2<\cdots<h_g$. 
We have shown that $h_1,\dots, h_g$ is a sequence of characteristic exponents of some 
irreducible quasi-ordinary Weierstrass polynomial $f_1$. By construction of this sequence
and by~(\ref{Eq:6'}) we get $\Delta_{T(f_1)}=\Delta_{T(f)}$. Hence by Theorem~\ref{Th:6}
$f$ is an irreducible quasi-ordinary Weierstrass polynomial. 
\end{proof}

\medskip

\noindent  Kiyek and Micus (\cite{K-M}) introduced the {\em semigroup} of an irreducible quasi-ordinary hypersurface $f(Y)=0$. Later Gonz\'alez P\'erez and Popescu-Pampu introduced again the semigroup  in their 
thesis (\cite{GP-tesis}, \cite{PP-tesis}), using different but equivalent definitions. This is the semigroup 
$ \deg f\bZ_{\geq 0}^d+\bZ_{\geq 0} \gamma_1+\cdots+\bZ_{\geq 0} \gamma_g,$
where $\gamma_1,\ldots, \gamma_g$ is the sequence defined in Theorem \ref{Theo1}.

\medskip

\noindent Since the Newton polytope $\Delta(D_f)$, for an irreducible quasi-ordinary polynomial $f(Y)$,  determines its semigroup, it also determines the sequence of characteristic exponents (see \cite{GP-tesis} and \cite{PP-tesis}). Observe that the proof of Theorem \ref{Th:6} gives us the sequence of characteristic exponents by using the equalities (8).

\medskip

\begin{Example}[\cite{Assi}, Example 1]
Consider $f_1 (Y)= Y^8-2X_1X_2Y^4+X_1^2X_2^2-X_1^3X_2^2 \in \bK[[X_1, X_2]][Y]$. We get 
$D_{f_1}(X_1,X_2,V)=-16777216(V-X_1^2X_2^2+X_1^3X^2)^3(V+X_1^3X_2^2)^4$, so $$\Delta(D_{f_1})=3\left\{\Teisssr{(2,2)}{1}{4}{2}\right\}
+4\left\{\Teisssr{(3,2)}{1}{6}{3}\right\}=\left\{\Teisssr{(6,6)}{3}{4}{2}\right\}+\left\{\Teisssr{(12,8)}{4}{6}{3}\right\}.$$

\noindent We get $H_0=1$, $H_1=4$, $H_2=8$, $\gamma_1=(2,2)$ and $\gamma_2=(12,8)$. We have $[W_1:W_0]=4=H_1/H_0$ and $[W_2:W_1]=2=H_2/H_1$, and we deduce that $f_1$ is irreducible.
\end{Example}

\medskip
\begin{Example}[\cite{Assi}, Example 2] 
Consider $f_2 (Y)= Y^8-2X_1X_2Y^4+X_1^2X_2^2-X_1^4X_2^2-X_1^5X_2^3 \in \bK[[X_1, X_2]][Y]$. We get  $D_{f_2}(X_1,X_2,V)=-16777216(V-X_1^2X_2^2+X_1^4X_2^2+X_1^5X_2^3)^3(V+X_1^4X_2^2+X_1^5X_2^3)^4$, so
$$\Delta(D_{f_2})=3\left\{\Teisssr{(2,2)}{1}{4}{2}\right\}+4\left\{\Teisssr{(4,2)}{1}{4}{2}\right\}=\left\{\Teisssr{(6,6)}{3}{4}{2}\right\}+\left\{\Teisssr{(16,8)}{4}{6}{3}\right\}.$$

\noindent We get $H_0=1$, $H_1=4$, $H_2=8$, $\gamma_1=(2,2)$ and $\gamma_2=(16,8)$. We have $[W_1:W_0]=4=H_1/H_0$ but $[W_2:W_1]=1\neq 2=H_2/H_1$, and we deduce that $f_2$ is  not irreducible.
\end{Example}

\medskip

\begin{Example}[\cite{Assi}] 

\noindent This is the Example 3\footnote{There is a typo in the equation of this example in \cite{Assi}. A. Assi communicated to us the right equation of this example.} in \cite{Assi}. Consider $f_3 (Y)= Y^8-2X_1X_2Y^4+X_1^3X_2^2-X_1^3X_2^5 \in \bK[[X_1, X_2]][Y]$. We get  $D_{f_3}=
-16777216(V+X_1^2X_2^2-X_1^3X_2^2+X_1^3X_2^5)^4(V-X_1^3X_2^2+X_1^3X_2^5)^3$, so
$$\Delta(D_{f_3})=4\left\{\Teisssr{(2,2)}{1}{4}{2}\right\}+3\left\{\Teisssr{(3,2)}{1}{4}{2}\right\}=\left\{\Teisssr{(8,8)}{4}{4}{2}\right\}+\left\{\Teisssr{(9,6)}{3}{4}{2}\right\}.$$

\noindent We get $H_0=1$, $H_1=5$, $H_2=8$. Thus $H_2/H_1$ is not an integer number, 
so $[W_2:W_1]\neq H_2/H_1$  and we deduce that $f_3$ is  not irreducible.
\end{Example}

\medskip

\begin{Remark} 
In general  $\Delta(D_f)$ does not determine $T(f)$ as shown in \cite{Eggers} and \cite[Proposition 2.2]{Lenarcik}. But in the above examples it does.
To obtain the tree models it is enough to remember that $\Delta(D_f)=\Delta_{T(f)}$ and use Definition \ref{page2} and Theorem~\ref{Th:3}. 
The appropriate tree models with indicated heights of bars are drawn below:

$$\mbox
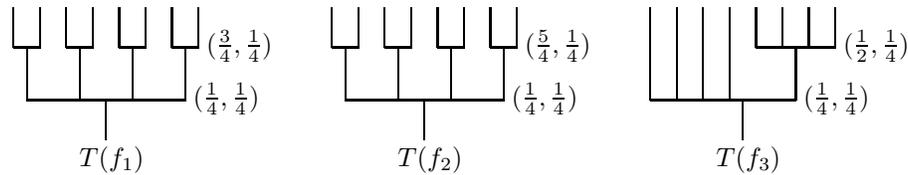
{%
\begin{picture}(100,60)(0,0)
\put(35,0){\line(0,1){15}} 
\put(25,-10){$T(f_1)$}
{\thicklines \put(5,15){\line(1,0){60}}} \put(68,13){$(\frac14,\frac14)$}
\put(5,15){\line(0,1){20}} 
{\thicklines \put(0,35){\line(1,0){10}}} 
\put(0,35){\line(0,1){15}} 
\put(10,35){\line(0,1){15}} 
\put(25,15){\line(0,1){20}} 
{\thicklines \put(20,35){\line(1,0){10}}} 
\put(20,35){\line(0,1){15}}
\put(30,35){\line(0,1){15}}
\put(45,15){\line(0,1){20}} 
{\thicklines \put(40,35){\line(1,0){10}}} 
\put(40,35){\line(0,1){15}}
\put(50,35){\line(0,1){15}}
\put(65,15){\line(0,1){20}}
{\thicklines \put(60,35){\line(1,0){10}}} \put(73,33){$(\frac34,\frac14)$}
\put(60,35){\line(0,1){15}} 
\put(70,35){\line(0,1){15}} 
\end{picture}} \qquad
\mbox{%
\begin{picture}(100,60)(0,0)
\put(35,0){\line(0,1){15}} 
\put(25,-10){$T(f_2)$}
{\thicklines \put(5,15){\line(1,0){60}}} \put(68,13){$(\frac14,\frac14)$}
\put(5,15){\line(0,1){20}} 
{\thicklines \put(0,35){\line(1,0){10}}} 
\put(0,35){\line(0,1){15}} 
\put(10,35){\line(0,1){15}} 
\put(25,15){\line(0,1){20}} 
{\thicklines \put(20,35){\line(1,0){10}}} 
\put(20,35){\line(0,1){15}}
\put(30,35){\line(0,1){15}}
\put(45,15){\line(0,1){20}} 
{\thicklines \put(40,35){\line(1,0){10}}} 
\put(40,35){\line(0,1){15}}
\put(50,35){\line(0,1){15}}
\put(65,15){\line(0,1){20}}
{\thicklines \put(60,35){\line(1,0){10}}} \put(73,33){$(\frac54,\frac14)$}
\put(60,35){\line(0,1){15}} 
\put(70,35){\line(0,1){15}} 
\end{picture}} \qquad
\mbox{%
\begin{picture}(100,60)(0,0)
\put(35,0){\line(0,1){15}} 
\put(25,-10){$T(f_3)$}
{\thicklines \put(0,15){\line(1,0){55}}} \put(58,13){$(\frac14,\frac14)$}
\put(0,15){\line(0,1){35}} 
\put(10,15){\line(0,1){35}} 
\put(20,15){\line(0,1){35}} 
\put(30,15){\line(0,1){35}} 
\put(55,15){\line(0,1){20}}
{\thicklines \put(40,35){\line(1,0){30}}} \put(73,33){$(\frac12,\frac14)$}
\put(40,35){\line(0,1){15}} 
\put(50,35){\line(0,1){15}} 
\put(60,35){\line(0,1){15}} 
\put(70,35){\line(0,1){15}} 
\end{picture}} $$

\medskip

\end{Remark}

\section{Discriminant of a $Y$-regular power series}
 \label{qo-series}
 
 \noindent In this section we generalize the notion of the discriminant $D_f(\underline{X},V)$, 
which was previously 
defined for Weierstrass polynomials, to an arbitrary $Y$-{\em regular power} series. 

\medskip

\noindent 
We say that a power series $f(\underline{X},Y)\in\bK[[\underline{X},Y]]$ is $Y$-regular of order $n$ if 
$f(0,Y)=cY^n+\hbox{\em higher order terms}$ with $c\neq0$. 

\medskip

\noindent Assume that $f\in\bK[[\underline{X},Y]]$ is $Y$-regular of order $n$. 
By Weierstrass preparation theorem 
for every $g\in\bK[[\underline{X},Y,V]]$ there exist a unique $q\in\bK[[\underline{X},Y,V]]$ and $a_0, \dots, a_{n-1}\in\bK[[\underline{X},V]]$ such that 
\[ g=(f-V)q+\sum_{i=0}^{n-1}a_iY^i .\]

\medskip

\noindent It follows that the quotient ring $A=\bK[[\underline{X},Y,V]]/(f-V)$ is a free $\bK[[\underline{X},V]]-$ module which admits the basis 
$1$, $\overline Y$, \dots, $\overline Y^{n-1}$, where $\overline Y$ is the coset of $Y$ in $A$. 
Let $\Phi_g:A\to A$ be an $\bK[[\underline{X},V]]$-endomorphism  induced by the multiplication 
$\bK[[\underline{X},Y,V]]\ni h\to g h \in \bK[[\underline{X},Y,V]]$. 

\medskip
\noindent
We put by definition ${\bD}_f(\underline{X},V)=\det \Phi_{\frac{\partial f}{\partial Y}}$.

\begin{Property}\ 
\label{orion}
\begin{itemize}
\item[(i)] If $f(\underline{X},Y)$ is a Weierstrass polynomial in the variable $Y$
               then $\bD_f(\underline{X},V)$ is equal to $D_f(\underline{X},V)$.
\item[(ii)] $\bD_f(\underline{X},V)$ belongs to the ideal $I=\Bigl(f-V, \frac{\partial f}{\partial Y}\Bigr)\bK[[\underline{X},Y,V]]$. Moreover 
                the radicals of the  ideals $(\bD_f)\bK[[\underline{X},V]]$ and $I\cap \bK[[\underline{X},V]]$ are the same.
\item[(iii)] Let $g(T,Y)=f(T^{c_1},\dots,T^{c_d},Y)$. Then $\bD_g(T,V)=\bD_f(T^{c_1},\dots,T^{c_d},V)$.
\item[(iv)] If $f(X,Y)\in\bK[[X,Y]]$ is a $Y$-regular power series in two variables and 
$\frac{\partial f}{\partial Y}(X,Y)=u(X,Y)\prod_{i=1}^{n-1}[Y-Y_i(X)]$
is a Newton-Puiseux factorization of 
its partial derivative then $\bD_f(X,V)=u'(X,V)\prod_{i=1}^{n-1}[f(X,Y_i(X))-V]$
where $u'(X,V)$ is a unity in $\bK[[X,V]]$.
\item[(v)] If $f(X,Y)\in\bC\{X,Y\}$ then $\bD_f(u,v)=0$ is an equation of the discriminant curve 
of the holomorphic mapping germ $(\bC^2,0)\to(\bC^2,0)$, $(u,v)=(x,f(x,y))$
in the sense of Casas-Alvero~\cite{Casas-Asian}.
\end{itemize}
\end{Property}

\noindent \begin{proof}
\begin{itemize}
\item[(i)] Let $n$ be the $Y$-degree of $f$. 
Then the $Y$-discriminant of $f-V$ is the determinant 
of the matrix of $\Phi_{\frac{\partial f}{\partial Y}}$ 
with respect to the basis $1$, $\overline Y$, \dots, $\overline Y^{n-1}$ 
(see \cite{Risler}, Appendix D.3.6).

\item [(ii)] The mapping  $\Phi:=\Phi_{\frac{\partial f}{\partial Y}}$ induces the exact sequence

\[
A\stackrel{\Phi}{\longrightarrow} A\longrightarrow \bK[[\underline{X},Y,V]]/I \longrightarrow 0.
\]

\noindent By definition (see 
\cite{Greuel}, Section 7.2), $(\bD_f)\bK[[\underline{X},V]]$ is the $0$-th Fitting ideal of the $\bK[[\underline{X},V]]$-module $\bK[[\underline{X},Y,V]]/I $. On the other hand $I\cap \bK[[\underline{X},V]]$ is the annihilator of $\bK[[\underline{X},Y,V]]/I$. By Proposition 20.6 of \cite{Eisenbud} (see also \cite{Greuel}, Exercise 7.2.5), we get the equality of the radicals.

\item[(iii)] Suppose that $f$ is $Y$-regular of order $n$. If 
$$Y^i\frac{\partial f}{\partial Y}=\sum_{j=0}^{n-1} m_{ij}(X_1,\ldots,X_d,V)Y^j+h_i(\underline{X},Y,V)(f(\underline{X},Y)-V)$$ then 
$$Y^i\frac{\partial g}{\partial Y}=\sum_{j=0}^{n-1} m_{ij}(T^{c_1},\ldots,T^{c_d},V)Y^j+h_i(T^c,Y,V)(g(T,Y)-V).$$
These relations, for $i=0,\ldots,n-1$, imply that $\bD_f(T^{c_1},\dots,T^{c_d},V)=\det (m_{ij}(T^{c_1},\dots,T^{c_d},V)_{n\times n})$
is equal to $\bD_g(T,V)$.

\item[(iv)] Suppose that $Y_i(X)$ are power series for $i=1,\ldots,n-1$. Since $\Phi_{gh}=\Phi_{g}\circ \Phi_{h}$ we get
$\bD_f(X,V)=\det \Phi_{\frac{\partial f}{\partial Y}}=
\det \Phi_{u(X,Y)}\prod_{i=1}^{n-1}\det \Phi_{Y-Y_i(X)}$. 
Moreover $\det \Phi_u \cdot \det \Phi_{u^{-1}}=\det (\hbox{\rm id})=1$. The substitution of $Y_i(X)$ for $Y$   determines an isomorphism between the $\bK[X,V]$-modules  $\bK[[X,Y,V]]/(f(X,Y)-V, Y-Y_i(X))$ and 
$\bK[[X,V]]/(f(X,Y_i(X))-V)$.
Hence the ideal generated by $\det \Phi_{Y-Y_i(X)}$, which is the $0$-Fitting ideal of both modules, is equal to $(f(X,Y_i(X))-V)\bK[[X,V]]$. The proof in this case is finished.

\medskip

\noindent  Let us consider  the general situation.
 There exists a natural number $m$ such that $\frac{\partial f}{\partial y}(T^m,Y)=u(T^m,Y) \prod_{i=1}^{n-1}(Y-Y_i(T^m))$ is a factorization in $\bK[[T,Y]]$. Using $(iii)$ and applying $(iv)$, in the case proved before,  to $g(T,Y):=f(T^m,Y)$ we get
 \begin{equation}
 \label{emu}
\bD_f(T^m,V)=\bD_g(T,V)=u'(T,V)\prod_{i=1}^{n-1}(f(T^m,Y_i(T^m))-V).
 \end{equation}
 
 \noindent By definition $\bD_f(T^m,V)\in \bK[[T^m,V]]$. Denote by $P(T,V)$ the product $\prod_{i=1}^{n-1}(f(T^m,Y_i(T^m))-V)$ appearing in~(\ref{emu}).
Let $\epsilon\in \bK$ be an $m$-th primitive root of unity.
Since $Y_i(T^m) \to Y_i((\epsilon \,T)^m)$ is a permutation of the roots of the derivative of $g$, we have $P(\epsilon \, T,V)=P(T,V)$, and consequently $P(T,V)\in \bK[[T^m,V]]$.
 
 We claim that $u'(T,V)=u''(T^m,V)$ for some $u''\in \bK[[X,V]]$. Indeed substituting $\epsilon \, T$ for $T$ in (\ref{emu}) 
  we get $u'(\epsilon \, T, V)=u'(T,V)$ which shows that $u'(T,V)\in \bK[[T^m,V]]$. 
 We get
 $\bD_f(X,V)=u''(X,V)\prod_{i=1}^{n-1}(f(X,Y_i(X))-V)$.

\item[(v)] The formula in $(iv)$ determines the equation of the discriminant curve in the sense of Casas-Alvero (see \cite{Kodai}, Lemma 4.5 in Appendix).
\end{itemize}
\end{proof}

\medskip

\noindent Remark that $\bD_f(\underline{X},V)$ extends, in a natural way, the definition of $D_f(\underline{X},V)$. 

\medskip

\begin{Theorem}
\label{TTT}
Let $f_1(\underline{X},Y)\in \bK[[\underline{X}]][Y]$ 
be a Weierstrass polynomial and let 
$f_2(\underline{X},Y)=u(\underline{X},Y)f_1(\underline{X},Y)$, 
where $u(\underline{X},Y)$ is a unit in $\bK[[\underline{X},Y]]$. 
Then the Newton polytopes of $D_{f_1}$ and $\bD_{f_2}$ are equal.
\end{Theorem}

\noindent \begin{proof}
Consider the substitution $g_i(T,Y)=f_i(T^{c_1},\ldots, T^{c_d},Y)$ for $i=1,2$.  Later on we assume that
$c_j\geq \deg f_1$ for $j=1,\ldots,d$. 

\noindent 
By  item (i) of Property~\ref{orion} we have $\bD_{f_1}=D_{f_1}$ and  $\bD_{g_1}=D_{g_1}$.
By Corollary 5.3 in \cite{GB-Gwo} and Property~\ref{orion}~(v)
we get $\Delta(\bD_{g_1})=\Delta(\bD_{g_2})$. In \cite{GB-Gwo} the above equality was proved in the convergent power series case. Anyway the methods in \cite{GB-Gwo} also work  for formal power series.

\noindent We finish the proof proceeding as in the proof of Theorem \ref{Th:2} replacing $\Delta(D_{f})$ by $\Delta(\bD_{f_2})$,  
$\Delta(D_g)$ by  $\Delta(\bD_{g_2})$, 
$\Delta_{T(f)}$ by $\Delta(\bD_{f_1})$ and 
$\Delta_{T(g)}$ by  $\Delta(\bD_{g_1})$.
\noindent The only difference is that we need to choose a vector $c=(c_1,\dots, c_d, c_{d+1})$
more carefully to assure that the hyperplanes 
$H_i = \{x\in\bR^{d+1} : \scalar{c}{x}=l(c,\Delta(\bD_{f_i}))\}$
support the Newton polyhedra  $\Delta(\bD_{f_i})$ at exactly
one point, for $i=1,2$.
\end{proof}

\begin{Corollary}
\label{cccoo}
Let $w(Y)$ be the Weierstrass polynomial of a $Y$-regular 
power series $f\in\bK[[X_1,\dots,X_d,Y]]$.
Then the following conditions are equivalent:

(i) the polynomial $w(Y)$ is quasi-ordinary,

(ii) the polytope $\Delta(D_w)\cap\bR^d\times\{0\}$ has only one vertex,

(iii) the polytope $\Delta(\bD_f)\cap\bR^d\times\{0\}$ has only one vertex,

(iv) $\bD_f(\underline{X},0)=u(\underline{X}) \cdot \mbox{monomial}$, where $u(0)\neq 0$.
\end{Corollary}

\noindent  \begin{proof} 
The Newton polytope of a series $h\in\bK[[\underline{X}]]$ has only one vertex if and only if 
$h$ has a form $u(\underline{X}) \cdot \mbox{\em monomial}$, where $u(0)\neq 0$. 
Since $\Delta(D_w)\cap\bR^d\times\{0\}$ is the Newton polytope of $D_w(\underline{X},0)$ 
and likewise $\Delta(\bD_f)\cap\bR^d\times\{0\}$ is the Newton polytope of $\bD_f(\underline{X},0)$, 
we get equivalences (i)$\Leftrightarrow$(ii) and (iii)$\Leftrightarrow$(iv). 
The equivalence (ii)$\Leftrightarrow$(iii) follows from Theorem~\ref{TTT}
\end{proof}

\medskip

\noindent  We call a $Y$-{\em regular power series} $f$ {\em quasi-ordinary} if it satisfies any of equivalent conditions (i)--(iv) of Corollary \ref{cccoo}.
We follow here Lipman who used~(i) in~\cite{Lipman} as a definition of quasi-ordinary convergent power series with complex coefficients.

\medskip 
\noindent Using Theorem~\ref{TTT} we may generalize main results of this paper, that is: Theorem~\ref{Th:2}, Corollary~\ref{C:2}, Theorem~\ref{Th:6} and Theorem~\ref{Theo1}, to $Y$-regular quasi-ordinary power series. 

\appendix
\section{Appendix: Computing indices}
\label{appendix}
\noindent Let $M \subset L$ be lattices in $\bZ^d$, i.e. additive subgroups of $\bZ^d$. 
In this appendix we recall a method of computing the 
{\em index} of ${M}$ in ${N}$.
By definition the index
$[N:M]$ is the cardinality of the quotient group $N/M$.  
Since  $[\bZ^d:N]\cdot [N:M]=[\bZ^d:M]$ it is enough to compute $[\bZ^d:M]$ and 
$[\bZ^d: N]$. The next theorem says how to do it by means of determinants.

\begin{Theorem}
Let $M=\bZ v_1+\cdots+\bZ v_n$ be a sub-lattice of $\bZ^d$ of finite index. Then $[\bZ^d:M]$ is the 
greatest common divisor of minors of maximal size of the matrix build from vectors $v_1,\dots,v_n$. 
\end{Theorem}

\medskip\noindent
\begin{proof} 
Let $\phi :\bZ^n\to \bZ^d$ be a group homomorphism given by 
$\phi(m_1,\dots,m_n)=m_1v_1+\cdots+m_nv_n$. 
Since every abelian group can be considered as a $\bZ$ module, 
this homomorphism induces the exact sequence of $\bZ$ modules
$$ \bZ^n\stackrel{\phi}{\to}\bZ^d\to \bZ^d/M\to0 .
$$ 

\noindent As in linear algebra we can associate with the mapping $\phi$ the matrix $A_{\phi}$ whose 
columns are the vectors $v_1,\dots,v_n$. 
The ideal generated in $\bZ$ by the minors of maximal size of $A_{\phi}$ is by definition the 
0-th Fitting ideal of the $\bZ$ module $\bZ^d/M$.  

\noindent To complete  the proof it is enough to show a general statement: for every finite abelian group $B$, treated as an $\bZ$ module,  number of elements of $B$ is the generator of 
the 0-th Fitting ideal of $B$.

\medskip

\noindent By the structure theorem for finitely generated abelian groups, $B$ is isomorphic to the direct sum $\bZ/q_1\bZ\oplus\cdots\oplus\bZ/q_s\bZ$ for some $q_1, \dots, q_s\in\bZ$. 
Thus $B$, treated as a $\bZ$ module, allows a finite presentation 
$$ \bZ^s\stackrel{\phi}{\to}\bZ^s\to B\to0
$$ 
where $\phi(n_1,\dots,n_s)=(q_1n_1,\dots,q_sn_s)$. 
Since $A_\phi$ is a square matrix, its determinant is the only minor of the maximal size. 
Thus the the 0-th Fitting ideal of $B$ is generated by $\det A_\phi$.
Notice that the determinant of a diagonal matrix $A_{\phi}$ is equal 
to the product $q_1\cdots\, q_s$ which is the cardinality of $B$.
\end{proof}

\medskip
\noindent
{\small Evelia Rosa Garc\'{\i}a Barroso\\
Departamento de Matem\'atica Fundamental\\
Facultad de Matem\'aticas, Universidad de La Laguna\\
38271 La Laguna, Tenerife, Espa\~na\\
e-mail: ergarcia@ull.es}

\medskip

\noindent {\small Janusz Gwo\'zdziewicz\\
Department of Mathematics\\
Technical University \\
Al. 1000 L PP7\\
25-314 Kielce, Poland\\
e-mail: matjg@tu.kielce.pl}
\end{document}